\documentclass[11pt]{article}

\usepackage[utf8]{inputenc}
\usepackage[T1]{fontenc}
\usepackage{amsmath,amsfonts,amssymb,amsthm}
\usepackage{stackrel}
\usepackage{enumerate}
\usepackage{graphicx}

\newtheorem{theorem}{Theorem}
\newtheorem{corollary}{Corollary}
\newtheorem{proposition}{Proposition}

\newtheorem{lemma}{Lemma}
\newtheorem{definition}{Definition}

\newtheorem*{remark}{Remark}

\begin{document}
\title{Two-dimensional volume-frozen percolation: exceptional scales}
\author{Jacob van den Berg, Pierre Nolin}
\date{CWI and VU University Amsterdam, ETH Z\"urich}
\maketitle

\begin{abstract}
We study a percolation model on the square lattice, where clusters ``freeze'' (stop growing) as soon as their volume (i.e. the number of sites they contain) gets larger than $N$, the parameter of the model. A model where clusters freeze when they reach \emph{diameter} at least $N$ was studied in \cite{BLN, Ki}. Using volume as a way to measure the size of a cluster -- instead of diameter -- leads, for large $N$, to a quite different behavior (contrary to what happens on the binary tree \cite{BKN}, where the volume model and the diameter model are ``asymptotically the same''). In particular, we show the existence of a sequence of ``exceptional'' length scales.
\end{abstract}

\section{Introduction}

\subsection{Frozen percolation}

Frozen percolation is a growth process on graphs that was first considered by Aldous \cite{Al} (motivated by sol-gel transitions), on the binary tree. It is a percolation-type process which can be described informally as follows. Let $G=(V,E)$ be a simple graph. Initially, all edges are closed, and they try to become open independently of each other. However, a connected component (cluster) of open edges is not allowed to grow forever: it stops growing as soon as it becomes infinite, which means that all edges along its boundary are then prevented from opening: we say that such a cluster ``freezes'', hence the name \emph{frozen percolation}.

Note that it is not clear at all that such a process exists. In \cite{Al}, Aldous studies the case when $G$ is the infinite binary tree, where each vertex has degree $3$ (and also the case of the planted binary tree, where all vertices have degree $3$, except the root vertex which has degree $1$). In this case, the tree structure allows for the derivation of recursion formulas, and it is shown in \cite{Al} that the frozen percolation process does exist. On the other hand, it was pointed out shortly afterwards by Benjamini and Schramm that for the square lattice $\mathbb{Z}^2$, such a process does not exist (see also Remark (i) after Theorem 1 in \cite{BT}).

A modification of the process, for which existence follows automatically from standard results, was introduced in \cite{BLN} by de Lima and the two authors. In this modified process, we stop the growth of a cluster when it reaches a certain ``size'' $N < \infty$, so that frozen percolation corresponds formally to $N = \infty$. The above mentioned non-existence result by Benjamini and Schramm motivated us to investigate what happens as $N \to \infty$. However, when $N$ is finite, one needs to make precise what ``size'' means, and in \cite{BLN}, the diameter of a cluster is used as a way to measure its size. This case was further studied by Kiss, who provides in \cite{Ki} a precise description of the process as $N \to \infty$. Roughly speaking, he shows that in a square of side length $KN$, for any fixed $K > 1$, only finitely many clusters freeze (with an exponential tail on their number), and they all do so in the near-critical window around the percolation threshold $p_c$. This implies in particular that the frozen clusters all look like near-critical percolation clusters: their density thus converges to $0$ as $N \to \infty$, as well as the probability for a given vertex to be frozen. In the final configuration, one only observes macroscopic non-frozen clusters, i.e. clusters with diameter of order $N$, but smaller than $N$.

In the case of the binary tree, it is shown in \cite{BKN} that the resulting configuration is completely different: one only observes frozen clusters (with diameter $\geq N$), and microscopic ones (with diameter $1,2,3, \ldots$), but no macroscopic non-frozen clusters. Moreover, it is explained that, on the tree, the size function does not really matter: under general conditions, the process converges to Aldous' process as $N \to \infty$. In the present paper, we go back to the case of two-dimensional lattices, and we show that measuring the size of clusters by their volume (i.e. the number of sites that they contain) leads to a behavior which is quite different from what happens in the diameter case. 

We now describe the process of interest, for some fixed parameter $N \geq 1$. We restrict ourselves to the square lattice $(\mathbb{Z}^2,\mathbb{E}^2)$, but note that our results would also hold on any two-dimensional lattice with enough symmetries, such as the triangular lattice or the honeycomb lattice. The set of vertices consists of points with integer coordinates, and two vertices $x$ and $y$ are connected by an edge (denoted by $x \sim y$) \emph{iff} $\|x-y\| = 1$, where $\|.\|$ refers to the usual Euclidean norm. We consider a collection of i.i.d. random variables $(\tau_e)_{e \in \mathbb{E}^2}$ indexed by the edges, where each $\tau_e$ is uniformly distributed on $[0,1]$. We work with a graph $G$ which is either the full lattice $\mathbb{Z}^2$, or a finite connected subgraph of it. The volume of a subset $A$ of $G$ is the number of sites that it contains, and we denote it by $|A|$. We start at time $0$ with all edges closed, and we let time increase. Each edge $e$ stays closed until time $\tau_e$, when it tries to become open: it is allowed to do so if and only if its two endpoints are in clusters of volume strictly smaller than $N$. Then, $e$ stays in the same state up to time $1$, when we can read the final configuration of the process. This process is well-defined, not only on finite subgraphs but also on $\mathbb{Z}^2$: indeed, it can be seen as a finite-range interacting particle system (the rate at which an edge becomes open only depends on the configuration within distance $N$ from that edge).

We denote by $\mathbb{P}_N^{(G)}$ the corresponding probability measure, and we drop the superscript $G$ when the graph used is clear from the context. Note also that the collection $(\tau_e)_{e \in \mathbb{E}^2}$ provides a natural coupling of the processes on different subgraphs of $\mathbb{Z}^2$: this observation is used repeatedly in our proofs.

The process could be described informally as follows. Starting from a configuration where all edges are closed, we let clusters grow as long as their volume is strictly smaller than $N$, and they stop growing (they ``freeze'') when their volume becomes at least $N$. We thus say that a given vertex $x$ is \emph{frozen} if it belongs to an open cluster with volume at least $N$ (and such a cluster is called ``frozen''). We are interested in the asymptotic behavior, as $N \to \infty$, of the probability that $0$ is frozen in the final configuration, i.e. at time $1$. We conjecture that
\begin{equation} \label{conj}
\mathbb{P}_N^{(\mathbb{Z}^2)}(\text{$0$ is frozen at time $1$}) \stackrel[n \to \infty]{}{\longrightarrow} 0.
\end{equation}
We present here some results related to this question, in the case of finite subgraphs of $\mathbb{Z}^2$.

We believe that frozen percolation provides an intriguing example of a non-monotone process, with competing effects that make it quite challenging to study. In particular, all our proofs require to follow, in some sense, the whole dynamics. The non-monotonicity appears clearly with the sequence of exceptional scales $(m_k)_{k \geq 1}$ in the results below.

\bigskip

\subsection{Statement of results}

In order to present our results, we need to introduce some more notation. We denote by $B(n) = [-n,n]^2$ the box of side length $2n$ centered at the origin (for $n \geq 0$). Our results pertain to the asymptotic behavior of
\begin{equation}\label{def:F}
F_N(n) = \mathbb{P}_N^{(B(n))}(\text{$0$ is frozen at time $1$})
\end{equation}
as $N \to \infty$, for some natural choices of $n \to \infty$ as a function of $N$.

For $p \in [0,1]$, $\mathbb{P}_p$ refers to independent bond percolation on $\mathbb{Z}^2$ with parameter $p$. It is a celebrated result of Kesten \cite{Ke_1980} that the percolation threshold is $p_c = \frac{1}{2}$ in this case, and we use the following notation for the one-arm probability at criticality:
\begin{equation}
\pi(n) = \mathbb{P}_{1/2}(0 \leadsto \partial B(n)).
\end{equation}

We are now in a position to state our main results. We start with one observation, that follows almost directly from standard results about independent percolation.

\begin{proposition} \label{prop1}
For every $C > 0$,
\begin{equation}
F_N \left( \left\lceil C \sqrt{N} \right\rceil \right) \stackrel[N \to \infty]{}{\longrightarrow} \phi(C),
\end{equation}
where $\phi(C) = \frac{1}{4C^2}$ for $C > \frac{1}{2}$, and $\phi(C) = 0$ for $C < \frac{1}{2}$.
\end{proposition}

Note that $\phi(C) \to 0$ as $C \to \infty$, which may tempt one to believe that for every function $g$ with $g(N) \gg \sqrt{N}$, one has $F_N(g(N)) \to 0$ as $N \to \infty$. However, the next theorem shows that one cannot naively exchange limits to say that in the full plane process, the probability for $0$ to be frozen converges to $0$. Indeed, $\sqrt{N}$ corresponds to a first scale $m_1(N)$ in a sequence $(m_k(N))_{k \geq 1}$ of exceptional scales, each of them leading to a non-trivial behavior. Roughly speaking, $m_2(N)$ is such that if we start with a box of this size, then a first ``giant'' cluster freezes and creates ``holes'' (here, by a hole, we mean a maximal connected component of unfrozen sites). The time at which this cluster freezes is such that the largest holes have size roughly $m_1(N)$, and most sites are in such holes. Then inside each of these holes, the process behaves similarly to the process in a box of size $m_1(N)$, so that a ``second-generation'' frozen cluster is produced. One can then define in the same way $m_3(N)$, $m_4(N)$, and so on. 

\begin{definition}
We define inductively the sequence of scales $(m_k(N))_{k \geq 0}$ by: $m_0 = 1$, and for all $k \geq 0$, $m_{k+1}=m_{k+1}(N)$ is given by
\begin{equation} \label{def_mk}
m_{k+1} = \left\lceil \left(\frac{N}{\pi(m_k)}\right)^{1/2} \right\rceil
\end{equation}
(i.e. $m_{k+1}^2 \pi(m_k) \simeq N$).
\end{definition}

Note that this definition immediately implies that
\begin{equation} \label{size_m1}
m_1(N) \sim c_0 \sqrt{N},
\end{equation}
for a certain constant $c_0>0$. Also,
\begin{lemma} \label{lemma_mk}
For every fixed $k \geq 0$, there exists $\alpha_k > 0$ such that: for $N$ large enough,
\begin{equation}
N^{\alpha_k} \leq \frac{m_{k+1}(N)}{m_k(N)} \leq N^{3^{-k}}.
\end{equation}
\end{lemma}

Theorem \ref{thm1} and \ref{thm2} below show that these scales $m_k$ ($k = 2, 3, \ldots$) are indeed exceptional.

\begin{theorem} \label{thm1}
Let $k \geq 2$ be fixed. For every $C \geq 1$, every function $\tilde{m}(N)$ that satisfies
\begin{equation}
C^{-1} m_k(N) \leq \tilde{m}(N) \leq C m_k(N)
\end{equation}
for $N$ large enough, we have
\begin{equation}
\liminf_{N \to \infty} F_N(\tilde{m}(N)) > 0.
\end{equation}
\end{theorem}

However, we do not expect these exceptional scales to correspond to a typical situation (i.e. sizes of holes produced in the full-plane process), so that the previous result does not contradict the conjecture. Moreover, the next theorem shows that if we start away from these unusual scales, then the probability for $0$ to be frozen converges to $0$ as expected.

\begin{theorem} \label{thm2}
For every integer $k \geq 0$ and every $\epsilon>0$, there exists a constant $C=C(k,\epsilon) \geq 1$ such that: for every function $\tilde{m}(N)$ that satisfies
\begin{equation}
C m_k(N) \leq \tilde{m}(N) \leq C^{-1} m_{k+1}(N)
\end{equation}
for $N$ large enough, we have
\begin{equation}
\limsup_{N \to \infty} F_N(\tilde{m}(N)) \leq \epsilon.
\end{equation}
\end{theorem}

Note that this theorem implies in particular the following result.

\begin{corollary}
Let $k \geq 0$ be fixed. If the function $\tilde{m}(N)$ satisfies $m_k(N) \ll \tilde{m}(N) \ll m_{k+1}(N)$ as $N \to \infty$, then
\begin{equation}
F_N(\tilde{m}(N)) \stackrel[N \to \infty]{}{\longrightarrow} 0.
\end{equation}
\end{corollary}

\bigskip

\section{Percolation preliminaries}

\subsection{Notations}

We first introduce some standard notation from percolation theory (for a more detailed account, we refer the reader to \cite{Gr,No_2008}). Two sets of vertices $A$ and $B$ are said to be connected, which we denote by $A \leadsto B$, if there exists an open path from some vertex of $A$ to some vertex of $B$. For a vertex $v$, $v \leadsto \infty$ means that $v$ lies in an infinite connected component, and we use
\begin{equation}
\theta(p) = \mathbb{P}_p(0 \leadsto \infty),
\end{equation}
which can also be seen as the density of the (unique) infinite cluster.

For a rectangle on the lattice of the form $R=[x_1,x_2] \times [y_1,y_2]$, we denote by $\mathcal{C}_V(R)$ (resp. $\mathcal{C}_H(R)$) the existence of a vertical (resp. horizontal) crossing, and we denote by $L(p) = L_{1/4}(p)$ the usual characteristic length defined in terms of crossings of rectangles:
\begin{equation}
\text{for all $p>\frac{1}{2}$,} \quad L_{1/4}(p) = \inf \left\{n \geq 1 \: : \: \mathbb{P}_p \left(\mathcal{C}_H([0,2n] \times [0,n]) \right) \geq \frac{3}{4} \right\},
\end{equation}
and $L_{1/4}(p) = L_{1/4}(1-p)$ for $p<\frac{1}{2}$.

We work with the dual graph of $\mathbb{Z}^2$, which can be seen as $(\frac{1}{2},\frac{1}{2}) + \mathbb{Z}^2$. We adopt the convention that a dual edge $e^*$ is open \emph{iff} the corresponding primal edge $e$ is closed (and we talk about dual-open and dual-closed paths).

We also denote by $A(n,n') = B(n') \setminus B(n)$ the annulus of radii $0 \leq n < n'$. Finally, if $\gamma$ is a circuit (i.e. a path whose vertices are all distinct, except the starting point and the end point, which coincide) on the dual lattice, we denote by $\mathcal{D}(\gamma)$ the domain (subgraph of $\mathbb{Z}^2$) obtained by considering the vertices (on the original lattice) which lie in the interior of $\gamma$, and keeping only the edges that connect two such vertices.

\bigskip

\subsection{Classical results}

We now collect useful results about independent percolation, which are needed for the proofs.
\begin{enumerate}[(i)]
\item Uniformly over $p>\frac{1}{2}$,
\begin{equation} \label{theta}
\theta(p) \asymp \pi(L(p))
\end{equation}
(where $\asymp$ means that the ratio between the two sides is bounded away from $0$ and $\infty$). This result is Theorem 2 in \cite{Ke_1987} (see also Corollary 41 in \cite{No_2008}, and the remark just below it).

\item There exist $\lambda_1, \lambda_2 >0$ such that: for all $p<p_c$, for all $k \geq 1$,
\begin{equation} \label{exp_decay}
\mathbb{P}_p(\mathcal{C}_V([0,2k] \times [0,k])) \leq \lambda_1 e^{-\lambda_2 \frac{k}{L(p)}}
\end{equation}
(see for instance Lemma 39 in \cite{No_2008}: it follows from similar arguments).

\item Let $(n_k)_{k \geq 1}$ be a sequence of integers, with $n_k  \stackrel[k \to \infty]{}{\longrightarrow} \infty$. If $(p_k)_{k \geq 1}$, with $\frac{1}{2} < p_k < 1$, satisfies $L(p_k) = o(n_k)$ as $k \to \infty$, then
\begin{equation} \label{largest_cluster}
\frac{|C_k|}{\theta(p_k) |B(n_k)|} \stackrel[k \to \infty]{\mathbb{P}}{\longrightarrow} 1,
\end{equation}
where $C_k$ denotes the largest $p_k$-open cluster in $B(n_k)$ (see Theorem 3.2 in \cite{BCKS}). Note that the condition $L(p_k) = o(n_k)$ is satisfied in particular when $p_k \equiv p \in (\frac{1}{2},1)$.

\item We will also use the following a-priori bounds on $\pi$: there exist $\alpha, c_1,c_2 > 0$ such that for all $1 \leq n_1 \leq n_2$,
\begin{equation} \label{apriori}
c_1 \left(\frac{n_2}{n_1}\right)^{\alpha} \leq \frac{\pi(n_1)}{\pi(n_2)} \leq c_2 \left(\frac{n_2}{n_1}\right)^{\frac{1}{2}}.
\end{equation}
The lower bound is a direct consequence of the Russo-Seymour-Welsh theorem, while the upper bound follows from the BK (van den Berg-Kesten) inequality. In particular, it yields immediately that $(n^2 \pi(n))_{n \geq 1}$ is essentially increasing, in the sense that:
\begin{equation} \label{ess_increasing}
\text{for all $1 \leq n_1 \leq n_2$,} \quad n_2^2 \pi(n_2) \geq c_2^{-1} n_1^2 \pi(n_1).
\end{equation}
\end{enumerate}

\bigskip

\subsection{Exceptional scales $(m_k)_{k \geq 1}$}

Let us first explain the intuitive meaning of the exceptional scales $(m_k)_{k \geq 1}$. Assuming that we know $m_k$, we look for a scale $m_{k+1}$ such that in a box of this size, a giant cluster appears, and freezes at a time $p_{k+1} = p_{k+1}(N)$ for which $L(p_{k+1}) \asymp m_k$, thus creating holes of size of order $m_k$. Since the largest cluster in a box of size $m$ has volume $\asymp \theta(p_{k+1}) m^2$, and $\theta(p_{k+1}) \asymp \pi(L(p_{k+1})) \asymp \pi(m_k)$ (from \eqref{theta}), we find that $m_{k+1}$ must satisfy the relation
$$\pi(m_k) m_{k+1}^2 \asymp N,$$
which leads to the inductive definition for $(m_k)_{k \geq 1}$.

Our proofs are based on the fact that we can identify, for square boxes, reasonably precisely the successive freezing times. On the one hand, if we start at scale $m_k$, then we expect $k-1$ successive freezings, that occur at times $p_j$ for which $L(p_j) \asymp m_{j-1}$ ($2 \leq j \leq k$), and a $k$th (and last) freezing at a supercritical time $p_1$. The last frozen cluster thus looks like a supercritical cluster, with density $\approx \theta(p_1)>0$, so that the probability for $0$ to be frozen is bounded away from $0$. On the other hand, if we start at a scale $\tilde{m}$ such that $m_k \ll \tilde{m} \ll m_{k+1}$, then we expect the next freezings to occur at times $p_j$ for which $m_{j-1} \ll L(p_j) \ll m_j$ ($2 \leq j \leq k$), and the last freezing at a time $p_1$ for which $1 \ll L(p_1) \ll m_1 \asymp \sqrt{N}$. We thus obtain a time $p_1$ which is well after the near-critical window (defined as the values of $p$ for which $L(p) \geq \sqrt{N}$), but that still satisfies $p_1 \to \frac{1}{2}$. Hence, the density of the last frozen cluster converges to $0$, as well as the probability for the origin to be frozen. Theorem \ref{thm1} and Theorem \ref{thm2} correspond, respectively, to these two possible situations.


\bigskip

We now prove Lemma \ref{lemma_mk}, which provides a-priori estimates on the scales $(m_k)_{k \geq 1}$.

\begin{proof}[Proof of Lemma \ref{lemma_mk}]
We proceed by induction over $k$. As noted earlier \eqref{size_m1}, the result holds for $k=0$. Assume it holds for some $k \geq 0$. By the definition of the $m_i$'s, we can write
$$\frac{m_{k+2}}{m_{k+1}} \stackrel[N \to \infty]{}{\sim} \left(\frac{\pi(m_{k})}{\pi(m_{k+1})}\right)^{1/2}.$$
Now, we can use the induction hypothesis and the a-priori bounds \eqref{apriori}.
\end{proof}

For future reference, let us also note that
\begin{equation}
m_k^2 \pi(m_k) = m_{k+1}^2 \pi(m_k) \left( \frac{m_k}{m_{k+1}} \right)^2 \sim N \left( \frac{m_k}{m_{k+1}} \right)^2
\end{equation}
so that for some strictly positive $\eta_k, \eta'_k \stackrel[k \to \infty]{}{\longrightarrow} 0$: for $N$ large enough,
\begin{equation} \label{apriori_mk}
N^{1-\eta_k} \leq m_k^2 \pi(m_k) \leq N^{1-\eta'_k}
\end{equation}
(in particular, it follows from \eqref{apriori} that $m_k \leq N^{2/3}$).

\begin{remark}
Informally speaking, \eqref{apriori_mk} means that $m_k$ approaches a scale $m_{\infty}$ that satisfies $m_{\infty}^2 \pi(m_{\infty}) = N^{1+o(1)}$, i.e. a scale such that: for critical percolation in a box of size $m_{\infty}$, the largest clusters have volume of order $N$. We expect $m_{\infty}$ to be the scale at which ``things start to happen'', i.e. at which the first frozen clusters form (looking like critical clusters). Then, successive clusters freeze around the origin, denser and denser, and we expect their number to tend to infinity (as $N \to \infty$).

However, we are not trying in this paper to make this heuristic argument rigorous. Instead, we are dealing with the process started in boxes of size $m_k$, for fixed $k \geq 1$. We hope that in order to study the last frozen cluster around the origin, so in particular our conjecture \eqref{conj}, it is enough to start around such a scale $m_k$, and understand what happens when $k \to \infty$. We study this idea further, for site percolation on the triangular lattice, in a second paper with Demeter Kiss \cite{BKN2}. Roughly speaking, we expect the full-plane process to ``fall'' between two exceptional scales, and not exactly on one of them. And moreover, we expect the lower bound in Theorem \ref{thm1} to decrease as $k \to \infty$: because of random effects on every scale, the process gets more and more ``spread out'', away from the exceptional scales.
\end{remark}

In the case of site percolation on the triangular lattice, where precise estimates were established thanks to the connection between critical percolation (in the scaling limit) and SLE (Schramm-Loewner Evolution) processes with parameter $6$ \cite{Sm_2001, LSW_2001a, LSW_2001b}, it is known \cite{LSW_2002} that
$$\pi(n) = n^{-\frac{5}{48}+o(1)} \quad \text{as $n \to \infty$}.$$
This leads immediately to $m_k(N) = N^{\delta_k + o(1)}$ as $N \to \infty$, where the sequence of exponents $(\delta_k)_{k \geq 0}$ satisfies
$$\delta_0 = 0, \quad \text{and for all $k \geq 0$,} \:\:\: \delta_{k+1} = \frac{1}{2} + \frac{5}{96} \delta_k.$$
In particular, this sequence is strictly increasing, and it converges to $\delta_{\infty} = \frac{48}{91}$ (note that $m_{\infty}(N) = N^{\delta_{\infty}}$ satisfies $m_{\infty}^2 \pi(m_{\infty}) = N^{1+o(1)}$). Let us also mention that for site percolation on the triangular lattice, many other critical exponents are known: in particular, $\theta(p) = (p-\frac{1}{2})^{5/36+o(1)}$ as $p \to \frac{1}{2}^+$, and $L(p) = |p-\frac{1}{2}|^{-4/3}$ as $p \to \frac{1}{2}$. Here, we decided to focus on bond percolation on the square lattice in order to stress that this more sophisticated technology is not needed for our results (Theorem \ref{thm1} and \ref{thm2}).

\bigskip

\section{Proofs of Proposition \ref{prop1} and Theorem \ref{thm1}}

\subsection{Proof of Proposition \ref{prop1}}

From now on, we drop the ceilings $\lceil . \rceil$ for notational convenience. We first note that the volume of the box $B(C \sqrt{N})$ is
\begin{equation}
| B(C \sqrt{N}) | \stackrel[N \to \infty]{}{\sim} 4 C^2 N,
\end{equation}
so that the case $C < \frac{1}{2}$ is clear. We can thus assume $C>\frac{1}{2}$, and introduce $\bar{p} \in (p_c,1)$ such that $\theta(\bar{p}) = \frac{1}{4 C^2}$, i.e.
\begin{equation}
\theta(\bar{p}) | B(C \sqrt{N}) | \stackrel[N \to \infty]{}{\sim} N.
\end{equation}

For arbitrary $\hat{p}$ and $\check{p}$ with $p_c < \check{p} < \bar{p} < \hat{p} < 1$, let us consider the following events:
\begin{itemize}
\item $D_1 = \{$there is a $\check{p}$-open path from $B(\sqrt[3]{N})$ to $\partial B(C \sqrt{N})$, and there is a $\check{p}$-open circuit in each of the annuli $A(k \sqrt[3]{N}, (k+1) \sqrt[3]{N})$, $k \geq 1$, contained in $B(C \sqrt{N})\}$,

\item $D_2 = \{$the largest $\check{p}$-open cluster in $B(C \sqrt{N})$ has volume $< N \}$,

\item $D_3 = \{$the largest $\hat{p}$-open cluster in $B(C \sqrt{N})$ has volume $> N \}$.
\end{itemize}
Each of these events has probability tending to $1$ as $N \to \infty$. For $D_1$, this follows from exponential decay of connection probabilities. For $D_2$ and $D_3$, it follows from the observation below \eqref{largest_cluster}, and our choice of $\bar{p}$: since $\theta$ is strictly increasing on $[p_c,1]$, we have
$$\theta(\check{p}) < \theta(\bar{p}) = \frac{1}{4 C^2} < \theta(\hat{p}).$$

If all these three events $D_1$, $D_2$ and $D_3$ hold, then there is no freezing in $B(C \sqrt{N})$ after time $\hat{p}$ (note that each annulus appearing in the definition of $D_1$ has volume $<N$, as well as $B(\sqrt[3]{N})$), which implies
\begin{equation}
F_N( C \sqrt{N} ) \leq \mathbb{P}_{\hat{p}} \Big( 0 \leadsto \partial B \Big(\frac{1}{3} \sqrt{N} \Big) \Big) + \mathbb{P}(D_1^c \cup D_2^c \cup D_3^c),
\end{equation}
and by letting $N \to \infty$,
\begin{equation} \label{upper_bound}
\limsup_{N \to \infty} F_N( C \sqrt{N} ) \leq \theta(\hat{p}).
\end{equation}

On the other hand, if each of $D_1$, $D_2$ and $D_3$ occurs, and if there is a $\check{p}$-open path from $0$ to $\partial B(C \sqrt{N})$, then $0$ freezes. Hence,
\begin{equation}
F_N( C \sqrt{N} ) \geq \mathbb{P}_{\check{p}}( 0 \leadsto \partial B (C \sqrt{N}) ) - \mathbb{P}(D_1^c \cup D_2^c \cup D_3^c),
\end{equation}
and by taking $N \to \infty$,
\begin{equation} \label{lower_bound}
\liminf_{N \to \infty} F_N( C \sqrt{N} ) \geq \theta(\check{p}).
\end{equation}

Since \eqref{upper_bound} and \eqref{lower_bound} hold for all $\hat{p} > \bar{p}$ and $\check{p} < \bar{p}$, we finally get, using the continuity of $\theta$,
\begin{equation}
F_N( C \sqrt{N} ) \stackrel[N \to \infty]{}{\longrightarrow} \theta(\bar{p}) = \frac{1}{4C^2},
\end{equation}
which completes the proof of Proposition \ref{prop1}.

\bigskip

\subsection{Proof of Theorem \ref{thm1}}

In order to proceed by induction, we prove Proposition \ref{prop2} below, of which Theorem \ref{thm1} is clearly a particular case. In order to state it, we first need more notation: for $n_1 < n_2$, let $\Gamma_N(n_1,n_2) = \{$for every dual circuit $\gamma$ in the annulus $A(n_1,n_2)$, for the process in the domain $\mathcal{D}(\gamma)$ with parameter $N$, $0$ is frozen$\}$. Here, we use the fact that the processes in various subgraphs of $\mathbb{Z}^2$ can be coupled in a natural way.

\begin{proposition} \label{prop2}
For any $k \geq 2$, and $0 < C_1 < C_2$, we have
\begin{equation}
\liminf_{N \to \infty} \mathbb{P}(\Gamma_N(C_1 m_k(N), C_2 m_k(N))) > 0.
\end{equation}
This result also holds for $k=1$ under the extra condition that $C_1 > (2 c_0)^{-1}$, where $c_0$ is the constant appearing in \eqref{size_m1}.
\end{proposition}

\begin{proof}
We prove the result by induction over $k$. The case $k=1$ follows from a similar reasoning as for Proposition \ref{prop1}. For that, let us introduce $p_c < \check{p} < \hat{p} < 1$ such that
\begin{equation}
0 < \theta(\check{p}) < \frac{1}{4 c_0^2 C_2^2} \quad \text{and} \quad \theta(\hat{p}) > \frac{1}{4 c_0^2 C_1^2}
\end{equation}
(we use $C_1 > (2 c_0)^{-1}$), and similar events as for the proof of Proposition \ref{prop1}:
\begin{itemize}
\item $D_1 = \{$there is a $\check{p}$-open path from $B(\sqrt[3]{N})$ to $\partial B(C_1 \sqrt{N})$, and there is a $\check{p}$-open circuit in each of the annuli $A(k \sqrt[3]{N}, (k+1) \sqrt[3]{N})$, $k \geq 1$, contained in $B(C_1 m_1(N))\}$,

\item $D_2 = \{$the largest $\check{p}$-open cluster in $B(C_2 m_1(N))$ has volume $< N \}$,

\item $D_3 = \{$the largest $\hat{p}$-open cluster in $B(C_1 m_1(N))$ has volume $> N \}$.
\end{itemize}
For the same reasons as before (see proof of Proposition \ref{prop1}), each of these events has probability tending to $1$ as $N \to \infty$. If $D_1$, $D_2$ and $D_3$ occur, and if there is a $\check{p}$-open path from $0$ to $\partial B(C_2 m_1(N))$, then for every dual circuit $\gamma$ in $A(C_1 m_1(N),C_2 m_1(N))$, $0$ freezes for the process in $\mathcal{D}(\gamma)$ with parameter $N$. This yields
\begin{equation}
\mathbb{P}(\Gamma_N(C_1 m_1(N), C_2 m_1(N))) \geq \mathbb{P}_{\check{p}}( 0 \leadsto \partial B (C_2 m_1(N)) ) - \mathbb{P}(D_1^c \cup D_2^c \cup D_3^c),
\end{equation}
and by taking $N \to \infty$,
\begin{equation} \label{lower_bound'}
\liminf_{N \to \infty} \mathbb{P}(\Gamma_N(C_1 m_1(N), C_2 m_1(N))) \geq \theta(\check{p}) > 0.
\end{equation}

Now, let us fix $k \geq 1$. We assume that Proposition \ref{prop2} holds for $k$, and we show that it then holds for $(k+1)$. In the following, $\gamma$ is an arbitrary dual circuit in $A(C_1 m_{k+1},C_2 m_{k+1})$.

We fix some $\delta>0$ very small (for the reasoning below, $\delta = \frac{1}{100}$ is enough, as the reader can check), and we define $p_2 = p_2(N)$ by
\begin{equation} \label{def_p2'}
\theta(p_2) (2 C_2 m_{k+1})^2 = N (1- \delta),
\end{equation}
and $p_1 = p_1(N)$ by
\begin{equation} \label{def_p1'}
\theta(p_1) \left( 2 \frac{9}{10} C_1 m_{k+1} \right)^2 = N (1+ \delta).
\end{equation}
Note that $p_2 \leq p_1$, and that the associated characteristic lengths satisfy
\begin{equation} \label{est_L}
L(p_1) \asymp L(p_2) \asymp m_k
\end{equation}
(the symbol $\asymp$ means that the constants depend only on $C_1$, $C_2$ and $\delta$). Indeed, it follows from \eqref{def_p1'} and \eqref{def_p2'}, and then \eqref{def_mk}, that for $i=1,2$,
$$\theta(p_i) \asymp \frac{N}{m_{k+1}^2} \asymp \pi(m_k),$$
so \eqref{theta} implies that $\pi(L(p_i)) \asymp \pi(m_k)$, which finally yields \eqref{est_L} (using \eqref{apriori}).

Let us now introduce the following events:
\begin{itemize}
\item $E_1 = \{$there is a $p_2$-open circuit in the annulus $A\left( \frac{7}{10} C_1 m_{k+1}, \frac{8}{10} C_1 m_{k+1} \right)$, and in the annulus $A\left( \frac{9}{10} C_1 m_{k+1}, C_1 m_{k+1} \right)\}$,

\item $E_2 = \{$the largest $p_1$-open cluster in the box $B\left( \frac{9}{10} C_1 m_{k+1} \right)$ has volume $\geq N \}$,

\item $E_3 = \{$the largest $p_1$-open cluster in $B\left( \frac{8}{10} C_1 m_{k+1} \right)$, the largest $p_1$-open cluster in $A\left( \frac{7}{10} C_1 m_{k+1}, C_1 m_{k+1} \right)$, and the largest $p_2$-open cluster in $B\left( C_2 m_{k+1} \right)$ all have volume $< N \}$.
\end{itemize}
Note that it follows from \eqref{est_L} and Lemma \ref{lemma_mk} that $L(p_i) \ll m_{k+1}$ ($i \in \{1,2\}$). Hence, \eqref{exp_decay} implies that $E_1$ has probability tending to $1$ as $N \to \infty$, and \eqref{largest_cluster} implies the same for $E_2$ and $E_3$.

We further introduce, for some $C_3>0$,
\begin{itemize}
\item $E_4 = \{$there is a $p_2$-open path from $B\left( 2 C_3 m_k \right)$ to $\partial B\left( C_1 m_{k+1} \right)$, and a $p_2$-open circuit in $A\left( 2 C_3 m_k, 4 C_3 m_k \right) \}$,

\item $E_5 = \{$there is a $p_1$-dual-open circuit in $A\left( C_3 m_k, 2 C_3 m_k \right) \}$.
\end{itemize}
The constant $C_3$ can be chosen arbitrarily (for instance, we can take $C_3 = 1$), except in the case $k=1$ when we have to assume that $C_3 > (2 c_0)^{-1}$ (in order to use the induction hypothesis). We know from \eqref{est_L} that $\mathbb{P}(E_4 \cap E_5) \geq \lambda > 0$, for some constant $\lambda$ independent of $N$ and the exact choice of $\gamma$.

\begin{figure}
\begin{center}
\includegraphics[width=10cm]{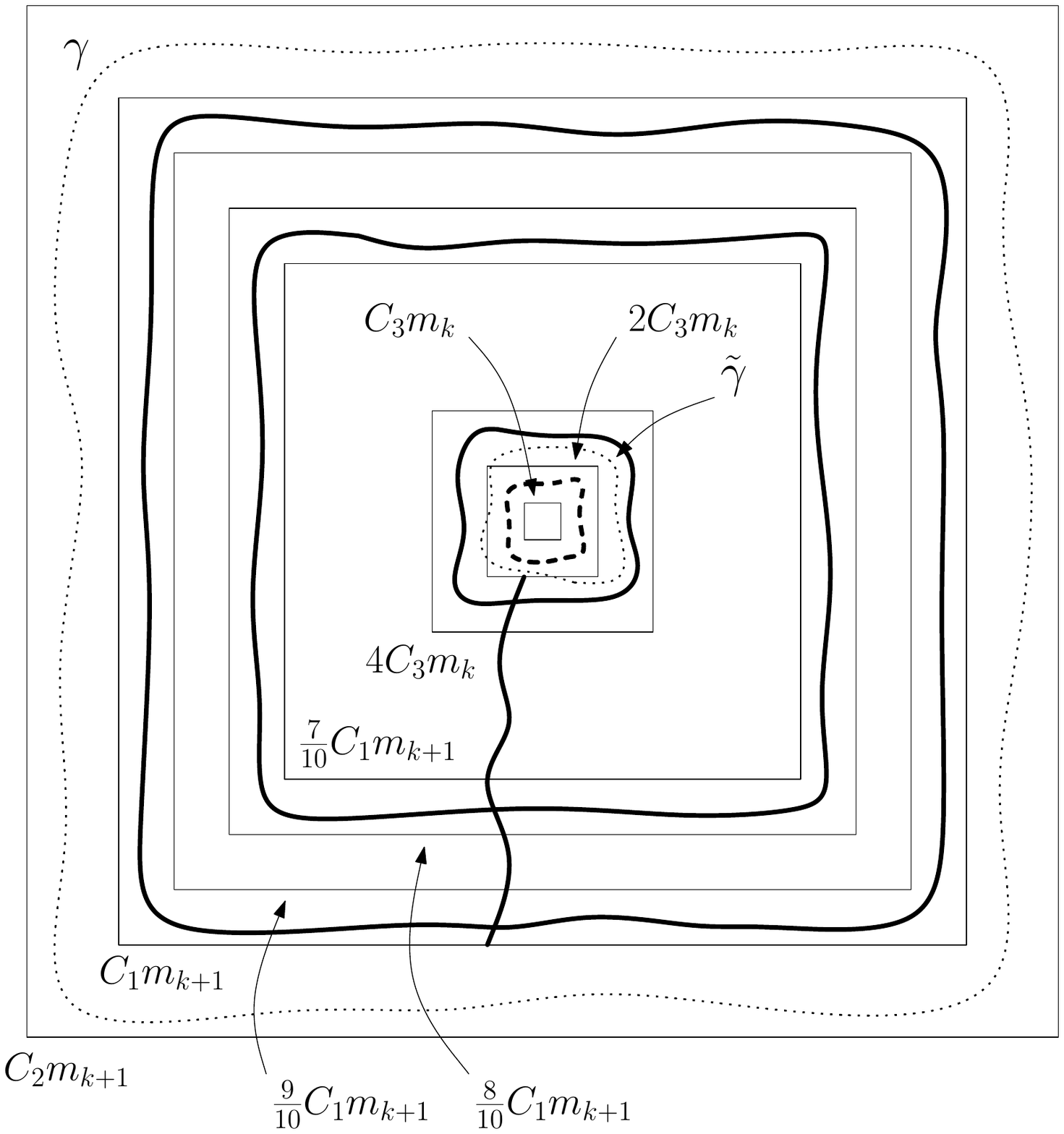}
\caption{\label{fig_thm1} This figure depicts the events $E_i$, $i = 1,\ldots, 5$, used in the proof of Proposition \ref{prop2}: solid lines represent $p_2$-open paths, and the small circuit in $A\left( C_3 m_k, 2 C_3 m_k \right)$ is $p_1$-dual-open. The dual circuit $\tilde{\gamma}$ is the boundary of the hole containing $0$.}
\end{center}
\end{figure}

We then make the following observations for the process in $\mathcal{D}(\gamma)$: if all these events $E_i$ ($1 \leq i \leq 5$) happen, then (no matter where the curve $\gamma$ is located exactly in $A(C_1 m_{k+1}, C_2 m_{k+1})$)
\begin{enumerate}[(i)]
\item the first time that vertices in the box $B\left( \frac{9}{10} C_1 m_{k+1} \right)$ freeze lies in the time interval $(p_2,p_1)$,

\item and if we look at the ``hole containing $0$'' in that frozen cluster, its boundary is a dual circuit contained in $A\left( C_3 m_k, 4 C_3 m_k \right)$.
\end{enumerate}
Hence, if we call $\Delta$ the event that (i) and (ii) occur, we get: for $N$ large enough,
$$\mathbb{P}(\Delta) \geq \frac{\lambda}{2} > 0.$$
We can then decompose $\Delta$ as the disjoint union
$$\Delta = \bigsqcup_{\tilde{\gamma}} \Delta(\tilde{\gamma}),$$
where $\tilde{\gamma}$ ranges over the set of dual circuits in $A\left( C_3 m_k, 4 C_3 m_k \right)$, and $\Delta(\tilde{\gamma})$ is the event $\{\Delta$ occurs, and $\tilde{\gamma}$ is the boundary of the hole mentioned in (ii)$\}$. For each such $\tilde{\gamma}$, we have
$$\Delta(\tilde{\gamma}) \supseteq \Delta_1(\tilde{\gamma}) \cap \Delta_2(\tilde{\gamma}),$$
where
\begin{itemize}
\item $\Delta_1(\tilde{\gamma})$ is the projection of the event $\Delta(\tilde{\gamma})$ on the set of edges in $\tilde{\gamma}$ and its exterior,

\item $\Delta_2(\tilde{\gamma}) = \{$the largest $p_1$-open cluster in $\mathcal{D}(\tilde{\gamma})$ has volume $< N \}$.
\end{itemize}
Note that $\Delta_1(\tilde{\gamma})$ and $\Delta_2(\tilde{\gamma})$ are independent, and $\mathbb{P}(\Delta_2(\tilde{\gamma})) \geq 1 - \epsilon(N)$, for some $\epsilon(N) \stackrel[N \to \infty]{}{\longrightarrow} 0$ (i.e. uniformly in $\tilde{\gamma}$ of the prescribed form). This yields
\begin{align*}
\mathbb{P}(\Gamma_N(C_1 m_{k+1}, C_2 m_{k+1})) & \\
& \hspace{-3cm} \geq \sum_{\tilde{\gamma}} \mathbb{P} \left( \Delta_1(\tilde{\gamma}) \cap \Delta_2(\tilde{\gamma}) \cap \{\text{$0$ freezes for the process in $\mathcal{D}(\tilde{\gamma})$}\}\right)\\
& \hspace{-3cm} \geq \sum_{\tilde{\gamma}} \mathbb{P} \left( \Delta_1(\tilde{\gamma}) \right) \mathbb{P} \left(\Delta_2(\tilde{\gamma}) \cap \{\text{$0$ freezes for the process in $\mathcal{D}(\tilde{\gamma})$}\}\right).
\end{align*}
For each $\tilde{\gamma}$, 
\begin{align*}
\mathbb{P} \left(\Delta_2(\tilde{\gamma}) \cap \{\text{$0$ freezes for the process in $\mathcal{D}(\tilde{\gamma})$}\}\right) & \\
& \hspace{-4cm} \geq \mathbb{P}(\Gamma_N(C_3 m_k, 4 C_3 m_k)) - \mathbb{P}(\Delta_2(\tilde{\gamma})^c),
\end{align*}
so
$$\mathbb{P}(\Gamma_N(C_1 m_{k+1}, C_2 m_{k+1})) \geq \mathbb{P}(\Delta) \big( \mathbb{P}(\Gamma_N(C_3 m_k, 4 C_3 m_k)) - \epsilon(N) \big).$$
Hence,
$$\liminf_{N \to \infty} \mathbb{P}(\Gamma_N(C_1 m_{k+1}, C_2 m_{k+1})) \geq \frac{\lambda}{2} \cdot \liminf_{N \to \infty} \mathbb{P}(\Gamma_N(C_3 m_k, 4 C_3 m_k)),$$
which completes the proof of Proposition \ref{prop2}, using the induction hypothesis.
\end{proof}


\bigskip

\section{Proof of Theorem \ref{thm2}}

We proceed in a similar way as for the proof of Theorem \ref{thm1}. We show by induction Proposition \ref{prop3} below, regarding the event $\tilde{\Gamma}_N(n_1,n_2) = \{$there exists a dual circuit $\gamma$ in the annulus $A(n_1,n_2)$ such that for the process in the domain $\mathcal{D}(\gamma)$ with parameter $N$, $0$ is frozen$\}$ ($n_1<n_2$). This result clearly implies Theorem \ref{thm2}. The proof turns out to be slightly simpler here, since no conditioning is required in order to use the induction hypothesis.

\begin{proposition} \label{prop3}
Let $k \geq 0$, $\epsilon>0$, and $0 < C_1 < C_2$. Then there exists a constant $C=C(k,\epsilon,C_1,C_2)$ such that: for every function $\tilde{m}(N)$ that satisfies
\begin{equation}
C m_k(N) \leq C_1 \tilde{m}(N) \leq C_2 \tilde{m}(N) \leq C^{-1} m_{k+1}(N)
\end{equation}
for $N$ large enough, we have
\begin{equation}
\limsup_{N \to \infty} \mathbb{P}(\tilde{\Gamma}_N(C_1 \tilde{m}(N), C_2 \tilde{m}(N))) \leq \epsilon.
\end{equation}
\end{proposition}

\begin{proof}
We proceed by induction over $k$. First, we note that the case $k=0$ is clear: we know from \eqref{size_m1} that $C^{-1} m_1(N) \sim C^{-1} c_0 \sqrt{N}$, and we just need to choose $C$ large enough so that $C^{-1} c_0 < \frac{1}{2}$ (then the corresponding probability is $0$ for $N$ large enough, since every domain $\mathcal{D}(\gamma)$ of the prescribed form has volume $<N$).

Now, let us fix $k \geq 0$. We assume that Proposition \ref{prop3} holds for $k$, and we show that it then holds for $(k+1)$. Let us consider an arbitrary $\epsilon>0$, and a function $\tilde{m}(N)$ such that
$$C^{(k+1)} m_{k+1}(N) \leq C_1 \tilde{m}(N) \leq C_2 \tilde{m}(N) \leq (C^{(k+1)})^{-1} m_{k+2}(N)$$
(we will explain later how to choose $C^{(k+1)}$). Finally, let us fix some $\delta>0$ very small (again, $\delta = \frac{1}{100}$ works).

We define $p_2 = p_2(N)$ by
\begin{equation} \label{def_p2}
\theta(p_2) (2 C_2 \tilde{m})^2 = N (1- \delta),
\end{equation}
and $p_1 = p_1(N)$ by
\begin{equation} \label{def_p1}
\theta(p_1) \left( 2 \frac{9}{10} C_1 \tilde{m} \right)^2 = N (1+ \delta)
\end{equation}
(note that $p_2 \leq p_1$). We first make two observations on the associated characteristic lengths.
\begin{enumerate}[(i)]
\item One has
\begin{equation} \label{est_L1}
L(p_1) \asymp L(p_2)
\end{equation}
(where the constants depend only on $C_1$, $C_2$ and $\delta$). Indeed, it follows from \eqref{def_p1} and \eqref{def_p2} that $\theta(p_1) \asymp \theta(p_2)$, so $\pi(L(p_1)) \asymp \pi(L(p_2))$ (using \eqref{theta}), and \eqref{apriori} finally implies \eqref{est_L1}.

\item Also,
\begin{equation} \label{est_L2}
L(p_1), L(p_2) \ll \tilde{m}.
\end{equation}
Indeed, we know that $\tilde{m} \leq m_{k+2}$ for $N$ large enough, so \eqref{ess_increasing} and \eqref{apriori_mk} imply that
$$\tilde{m}^2 \pi(\tilde{m}) \leq c_2 m_{k+2}^2 \pi(m_{k+2}) \leq c_2 N^{1-\eta'_{k+2}}.$$
Since we know from \eqref{def_p1} and \eqref{theta} that $\tilde{m}^2 \pi(L(p_i)) \asymp N$ ($i \in \{1,2\}$), we deduce $\pi(L(p_i)) \gg \pi(\tilde{m})$, and finally (using \eqref{apriori}) $L(p_i) \ll \tilde{m}$.
\end{enumerate}
We now consider the following events:
\begin{itemize}
\item $E_1 = \{$there is a $p_2$-open circuit in the annulus $A\left( \frac{7}{10} C_1 \tilde{m}, \frac{8}{10} C_1 \tilde{m} \right)$, and in the annulus $A\left( \frac{9}{10} C_1 \tilde{m}, C_1 \tilde{m} \right)\}$,

\item $E_2 = \{$the largest $p_1$-open cluster in the box $B\left( \frac{9}{10} C_1 \tilde{m} \right)$ has volume $\geq N \}$,

\item $E_3 = \{$the largest $p_1$-open cluster in $B\left( \frac{8}{10} C_1 \tilde{m} \right)$, the largest $p_1$-open cluster in $A\left( \frac{7}{10} C_1 \tilde{m}, C_1 \tilde{m} \right)$, and the largest $p_2$-open cluster in $B\left( C_2 \tilde{m} \right)$ all have volume $< N \}$.
\end{itemize}
Note that each of these events has probability tending to $1$ as $N \to \infty$: it follows from \eqref{est_L2}, combined with \eqref{exp_decay} (for $E_1$) and with \eqref{largest_cluster} (for $E_2$ and $E_3$).

We introduce now, for $C_3$, $C_4 > 0$ to be chosen later,
\begin{itemize}
\item $E_4 = \{$there is a $p_2$-open path from $B\left( C_3 L(p_2) \right)$ to $\partial B\left( C_1 \tilde{m} \right)$, and a $p_2$-open circuit in $A\left( C_3 L(p_2), 2 C_3 L(p_2) \right) \}$,

\item $E_5 = \{$there is a $p_1$-dual-open circuit in $A\left( C_4 L(p_2), C_3 L(p_2) \right) \}$.
\end{itemize}
It follows from \eqref{est_L1} that we can fix $C_3$ sufficiently large so that for $N$ large enough, $E_4$ occurs with probability at least $1 - \epsilon$, and then $C_4$ small enough to that for $N$ large enough, $E_5$ occurs with probability at least $1 - \epsilon$. Note that the two constants $C_3$ and $C_4$ are universal: they can be chosen independently of $N$ (and the precise choice of $\gamma$).

\begin{figure}
\begin{center}
\includegraphics[width=10cm]{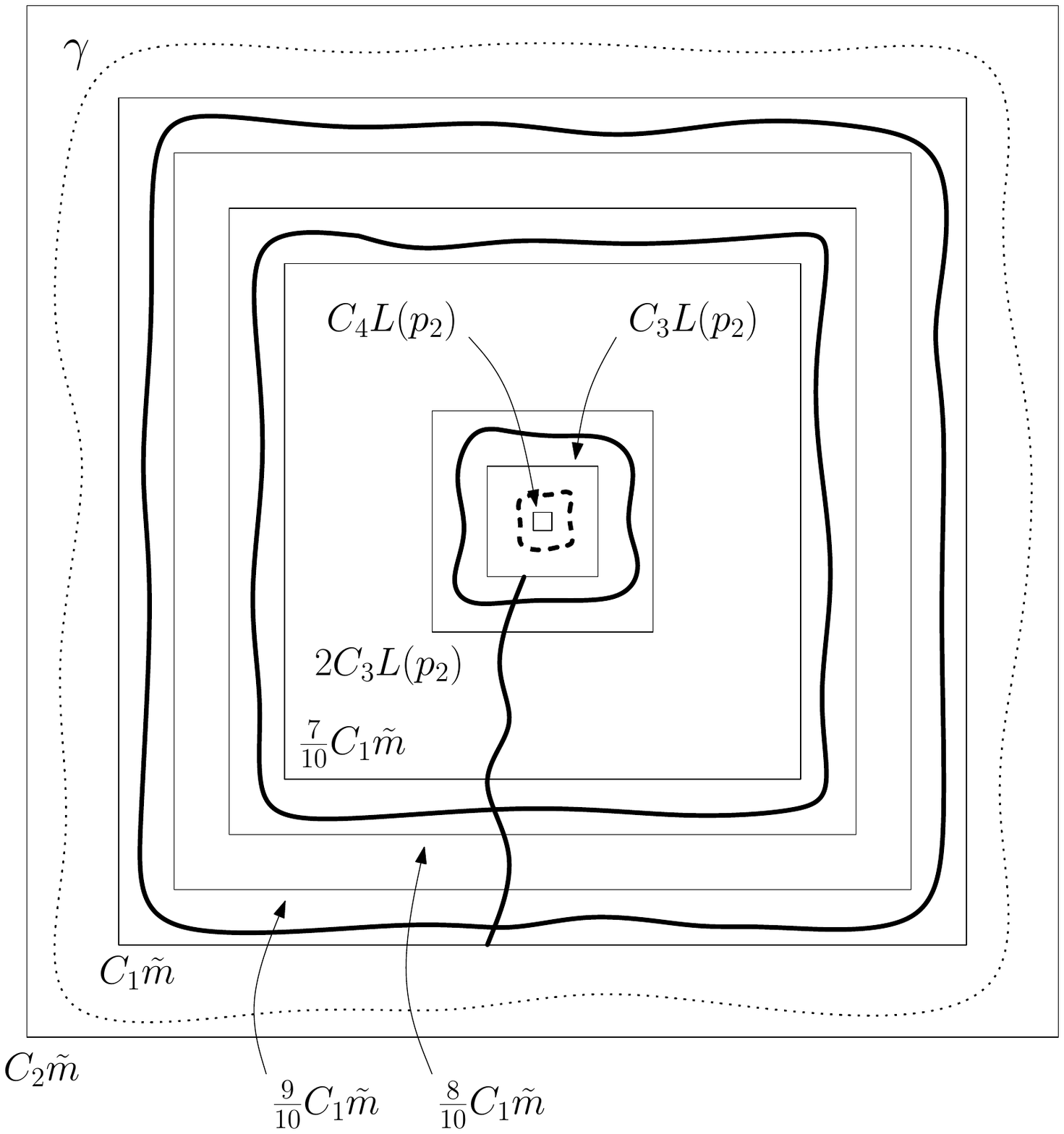}
\caption{\label{fig_thm2} This figure presents the construction used for the proof of Proposition \ref{prop3}: solid lines correspond to $p_2$-open paths, and the small circuit in dashed line is $p_1$-dual-open. Here, we need the ratio $C_3/C_4$ to be large enough.}
\end{center}
\end{figure}

We then make the following observations for the process in $\mathcal{D}(\gamma)$: if all these events $E_i$ ($1 \leq i \leq 5$) happen, then (no matter where the curve $\gamma$ is located exactly in $A(C_1 \tilde{m}, C_2 \tilde{m})$)
\begin{enumerate}[(i)]
\item the first time that vertices in the box $B\left( \frac{9}{10} C_1 \tilde{m} \right)$ freeze lies in the time interval $(p_2,p_1)$,

\item and if we look at the ``hole containing $0$'' in that frozen cluster, its boundary is a dual circuit contained in $A\left( C_4 L(p_2), 2 C_3 L(p_2) \right)$.
\end{enumerate}
This implies
$$\tilde{\Gamma}_N(C_1 \tilde{m}(N), C_2 \tilde{m}(N)) \cap \left( \bigcap_{1 \leq i \leq 5} E_i\right) \subseteq \tilde{\Gamma}_N(C_4 L(p_2), 2 C_3 L(p_2)).$$
Hence, using that $\limsup_{N \to \infty} \mathbb{P}\left(\left( \bigcap_{1 \leq i \leq 5} E_i\right)^c \right) \leq 2 \epsilon$, we get
$$\limsup_{N \to \infty} \mathbb{P}(\tilde{\Gamma}_N(C_1 \tilde{m}(N), C_2 \tilde{m}(N))) \leq \limsup_{N \to \infty} \mathbb{P}(\tilde{\Gamma}_N(C_4 L(p_2), 2 C_3 L(p_2))) + 2 \epsilon.$$
Now, we would like to apply the induction hypothesis to the right-hand side: for that, let us denote by $C^{(k)}$ the constant associated with $k$, $\epsilon$, and $0< C_4 < 2 C_3$. In order to be in a position to use the induction hypothesis, we need to show that $C^{(k+1)}$ can be chosen so as to ensure that
\begin{equation} \label{cond_ck}
C^{(k)} m_k(N) \leq C_4 L(p_2) \leq 2 C_3 L(p_2) \leq (C^{(k)})^{-1} m_{k+1}(N)
\end{equation}
for $N$ large enough. Indeed, this would then imply that
$$\limsup_{N \to \infty} \mathbb{P}(\tilde{\Gamma}_N(C_4 L(p_2), 2 C_3 L(p_2))) \leq \epsilon,$$
and complete the proof of Proposition \ref{prop3}. But \eqref{cond_ck} is satisfied for $C^{(k+1)}$ large enough: it follows immediately from \eqref{def_p2}, combined with \eqref{theta} and \eqref{apriori}.
\end{proof}

\bibliographystyle{plain}
\bibliography{FP_volume}

\end{document}